\documentclass[12pt,twoside,a4paper,reqno]{amsart}

\usepackage{graphicx}
\usepackage{graphics}
\usepackage{amssymb}
\usepackage{eucal}
\usepackage{amsmath}
\usepackage{amsthm}
\newtheorem{theorem}{Theorem}[section]
\newtheorem{lemma}[theorem]{Lemma}

\newtheorem{proposition}[theorem]{Proposition}
\newtheorem{corollary}[theorem]{Corollary}

\theoremstyle{definition}
\newtheorem{definition}[theorem]{Definition}
\newtheorem{example}[theorem]{Example}

\theoremstyle{remark}
\newtheorem{remark}[theorem]{Remark}

\numberwithin{equation}{section}

\begin{document}

\title{Characterization of  continuous $g$-frames \\via operators}

\author{Morteza Rahmani}
\address{Young Researchers and Elite Club, Ilkhchi Branch, Islamic Azad University,
Ilkhchi, Iran
}
%\curraddr{Department of Mathematics and Statistics,
%   Case Western Reserve University, Cleveland, Ohio 43403}
\email{morteza.rahmany@gmail.com}

%\author{Author Two}
%\address{Mathematical Research Section, School of Mathematical Sciences,
%Australian National University, Canberra ACT 2601, Australia}

%    General info

%\date{January 1, 2001 and, in revised form, June 22, 2001.}



\begin{abstract}
In this paper we introduce and show some new notions and results on $cg$-frames of Hilbert spaces.
We define $cg$-orthonormal bases for a Hilbert space $H$ and verify their properties and relations with $cg$-frames.
Actually, we present that every $cg$-frame can be represented as a composition of  a $cg$-orthonormal basis and an operator  under some conditions.
Also, we find for any $cg$-frame an induced $c$-frame and study their properties and relations.
% we show that every $c$-frame can be written as a (multiple of a) sum of three orthonormal mappings and
%we specify a  necessary and sufficient condition to represent a $c$-frame as a linear combination of two orthonormal mappings.
Moreover, we  show that every $cg$-frame can be written as  aggregate  of two Parseval $cg$-frames.
In addition, We show each $cg$-frame  as a summation of  a $cg$-orthonormal basis and a $cg$-Riesz basis.
\end{abstract}

\maketitle

\section{Introduction}

Frames (discrete frames) in Hilbert spaces were introduced by Duffin and Schaeffer \cite{Duff} in 1952 to study some deep problems in nonharmonic Fourier series. After the illustrious paper \cite{Daub} by Daubechies, Grossmann and Meyer, frame theory popularized immensely.

A frame for a Hilbert space  allows each vector in the space to be written as a linear
combination of the elements in the frame, but linear independence between
the frame elements is not required. Intuitively, a frame can be thought
as a basis to which one has added more elements.

Generally, frames have been used in signal processing, image processing, data compression and sampling theory.
%A discrete frame is a countable family of elements in a separable Hilbert space which allows for a stable, not necessarily unique, decomposition of an arbitrary element into an expansion of the frame elements.
Later, motivated by the theory of coherent states, this concept was generalized  to families indexed by some locally compact space endowed with a Radon measure. This approach leads to the notion of continuous frames \cite{[1], [2], [31], [43]}.
%Prominent examples are connected to the continuous wavelet transform \cite{[1],[38]} and the short time Fourier transform \cite{[34]}.
%In mathematical physics, these frames are referred to as coherent states \cite{[1],[37]}.
Some results about continuous frames and their generalizations can be found in \cite{pg-frame, bg-frame, pframe, RND}.

In this paper, inspired by \cite{Ra} and \cite{Guo},  we generalize some results to $cg$-frames.
\\The paper is organized as follows. In Section 2, we introduce the concept of $cg$-orthonormal bases for Hilbert spaces and discuss about their characteristics and their relations with $cg$-frames and  $c$-frames. Our aim in Section 3 is
 describing every continuous $g$-frame as a  sum of two Parseval continuous $g$-frames. We also present that every continuous $g$-frame can be written as a linear combination of an $cg$-orthonormal basis and a $cg$-Riesz basis.

\vspace{2mm}

Throughout this paper, $H$ is a separable Hilbert space, $(\Omega,\mu)$ is a measure space with positive measure $\mu$ and $\{H_{\omega}\}_{\omega \in\Omega}$ is a family of separable Hilbert spaces.
%Some times in some results we need $\mu$  to be $\sigma$-finite that it will be mentioned.

We first review the definition of continuous frames and continuous  $g$-frames.

\begin{definition}
{\rm(\cite{RND})} Suppose that $(\Omega,\mu)$ is a measure space with positive measure $\mu$. A mapping $f:\Omega\longrightarrow H$ is called a \emph{continuous frame}, or simply a \emph{$c$-frame}, with respect to $(\Omega,\mu)$ for $H$, if:
\\$(i)$ For each $h\in H$, $\omega\longmapsto \langle h,f(\omega)\rangle$ is a measurable function,
\\$(ii)$ there exist positive constants $A$ and $B$ such that
\begin{align}\label{df01}
A\|h\|^{2}\leq \int_{\Omega} |\langle h,f(\omega)\rangle|^{2} d\mu(\omega)\leq B\|h\|^{2},\quad h\in H.
\end{align}

The constants $A,B$ are called \emph{$c$-frame bounds}. If $A,B$ can be chosen such that $A=B$, then $f$ is called a \emph{tight $c$-frame} and if
$A = B = 1$, it is called a Parseval $c$-frame. A mapping $f$ is called \emph{$c$-Bessel mapping} if the second inequality in (\ref{df01}) holds.
In this case, $B$ is called the \emph{Bessel bound}.
\end{definition}

%We can define some operators associated to a $c$-Bessel mapping. The following proposition is a useful tool in the rest of our discussion.

Some operators associated to  $c$-Bessel mappings can be useful to characterize them.

\begin{proposition}\label{TT}
{\rm(\cite{RND})} Let  $(\Omega,\mu)$ be a measure space and $f:\Omega\longrightarrow H$ be a $c$-Bessel mapping for $H$. Then the operator
$ T_{f} : L^{2}(\Omega,\mu)\longrightarrow H$, weakly defined by
\begin{align}\label{T}
\langle T_{f}\varphi,h\rangle=\int_{\Omega}\varphi(\omega)\langle f(\omega),h\rangle d\mu(\omega),\quad h\in H,
\end{align}
is well defined, linear, bounded, and its adjoint is given by
\begin{align}\label{T*}
T^*_{f} : H\longrightarrow L^{2}(\Omega,\mu),\quad T^*_{f}h(\omega)=\langle h,f(\omega)\rangle,~\,\omega\in \Omega.
\end{align}
\end{proposition}
The operator $T_{f}$ is called  \emph{synthesis operator} and $T^*_{f}$ is called
\emph{analysis operator} of $f$.

If $f$ is a $c$-Bessel mapping with respect to $(\Omega,\mu)$ for $H$, then the operator
$ S_{f} : H\longrightarrow H$ defined by $S_{f}=T_{f}T^*_{f}$, is called \emph{frame operator} of $f$. Thus
$$\langle S_{f}h,k\rangle=\int_{\Omega}\langle h,f(\omega)\rangle\langle f(\omega),k\rangle d\mu(\omega),\quad h,k\in H.$$
 If $f$ is a $c$-frame for $H$, then $S$ is invertible.

The converse of above proposition holds when $\mu$ is $\sigma$-finite in the measure space  $(\Omega,\mu)$.

\begin{proposition}\label{2.7}
{\rm(\cite{RND})} Let $(\Omega,\mu)$ be a measure space where $\mu$ is $\sigma$-finite. Let $f:\Omega\longrightarrow H$ be a mapping such that for each $h\in H$, $\omega\longmapsto\langle h,f(\omega)\rangle$ is measurable. If the mapping $T_f:L^{2}(\Omega,\mu)\longrightarrow H$ defined by (\ref{T}), is a bounded operator, then $f$ is a $c$-Bessel mapping.
\end{proposition}

%The next theorem gives an equivalent characterization of a continuous frame.

\begin{theorem}\label{2.9}
{\rm(\cite{RND})} Suppose that $(\Omega,\mu)$ is a measure space where $\mu$ is $\sigma$-finite. Let $f:\Omega\longrightarrow H$ be a mapping such that for each $h\in H$, $\omega\longmapsto\langle h,f(\omega)\rangle$ is measurable. The mapping $f$ is a $c$-frame with respect to $(\Omega,\mu)$ for $H$ if and only if the operator $T_f:L^{2}(\Omega,\mu)\longrightarrow H$  defined by (\ref{T}), is a bounded and onto operator.
\end{theorem}

\begin{definition}
Let $\varphi\in\Pi_{\omega\in \Omega}H_\omega$. We say that $\varphi$ is \emph{strongly measurable} if $\varphi$ as
a mapping of $\Omega$ to $\oplus_{\omega \in \Omega}H_{\omega}$ is measurable, where
$$\Pi_{\omega\in \Omega}H_\omega=\big\{f:\Omega\longrightarrow \cup_{\omega\in \Omega} H_\omega~;~f(\omega) \in H_\omega\big\}.$$
\end{definition}

 Now, we review the definition of continuous $g$-frames.

\begin{definition}
We call $\{\Lambda_{\omega}\in B(H,H_\omega): \omega\in \Omega \}$ a \emph{continuous generalized frame}, or simply a   $cg$-frame, for $H$ with respect to $\{H_\omega\}_{\omega\in \Omega}$, if:
\\$(i)$ for each $f\in H$, $\{\Lambda_{\omega}f\}_{\omega\in \Omega}$ is strongly measurable,
\\$(ii)$ there are two positive constants $A$ and $B$ such that
\begin{align}\label{ccg}
A\|f\|^{2}\leq\int_{\Omega}\|\Lambda_\omega f\|^{2}d\mu(\omega)\leq B\|f\|^{2},\quad f\in H.
\end{align}
We call $A$ and $B$ the lower and upper  $cg$-frame bounds, respectively. If $A,B$ can be chosen such that $A=B$, then $\{\Lambda_{\omega}\}_{\omega\in \Omega}$ is called a \emph{tight $cg$-frame} and if
$A = B = 1$, it is called a Parseval $cg$-frame. A family $\{\Lambda_{\omega}\}_{\omega\in \Omega}$ is called \emph{$cg$-Bessel family} if the second inequality in (\ref{ccg}) holds.
\end{definition}

Now, let the space $\big(\oplus_{\omega \in \Omega}H_{\omega},\mu\big)_{L^2}\subseteq \Pi_{\omega\in \Omega}H_\omega$ be defined as follows,
\begin{align*}
\big(\oplus_{\omega \in \Omega}H_{\omega},\mu\big)_{L^2}=
\big\{\varphi |~\varphi~is~strongly~measurable, \int_{\Omega}\|\varphi(\omega)\|^{2}d\mu(\omega) < \infty\big\}.
\end{align*}
The space $\big(\oplus_{\omega \in \Omega}H_{\omega},\mu\big)_{L^2}$ is a Hilbert space with inner product
$$\langle\varphi,\psi \rangle=\int_{\Omega}\langle \varphi(\omega),\psi(\omega) \rangle d\mu(\omega).$$

\begin{proposition}\label{Tg}
{\rm(\cite{cg-frame})} Let $\{\Lambda_{\omega}\}_{\omega\in \Omega}$ be a $cg$-Bessel family for $H$ with respect to
$\{H_\omega\}_{\omega\in \Omega}$ with Bessel bound $B$. Then the mapping $T$ of $\big(\oplus_{\omega \in \Omega}H_{\omega},\mu\big)_{L^2}$ to $H$ defined by
\begin{align}\label{TT}
\langle T\varphi,h\rangle=\int_{\Omega} \langle \Lambda^{*}_{\omega}\varphi(\omega),h\rangle d\mu(\omega),\quad \varphi\in\big(\oplus_{\omega \in \Omega}H_{\omega},\mu\big)_{L^2},~h\in H,
\end{align}
is linear and bounded with $\|T\|\leq \sqrt{B}$. Furthermore for each $h \in H$ and $\omega \in \Omega$
\begin{align}
T^{*}(h)(\omega)=\Lambda_{\omega}h.
\end{align}
\end{proposition}
The operators $T$ and $T^*$ are called \emph{synthesis} and \emph{analysis}  operators of $cg$-Bessel family $\{\Lambda_{\omega}\}_{\omega\in \Omega}$, respectively.

\vspace{2mm}
Let $\{\Lambda_{\omega}\}_{\omega\in \Omega}$ be a $cg$-frame  for $H$ with respect to $\{H_\omega\}_{\omega\in \Omega}$
 with frame bounds $A,B$. The operator $S:H\longrightarrow H$ defined by
\begin{align}
\langle Sf,g\rangle=\int_{\Omega}\langle f,\Lambda_{\omega}^{*}\Lambda_{\omega}g \rangle d\mu(\omega),\quad f,g\in H,
\end{align}
is called the \emph{frame operator} of $\{\Lambda_{\omega}\}_{\omega\in \Omega}$ which is a positive  and invertible operator.

Now, we state a known result that is helpful in proving some results.

\begin{proposition}\label{06}
{\rm(\cite{sum2})} Let $K : H \longrightarrow H$ be a bounded linear operator. Then the following hold.
\\$(i)$ $K = \alpha(U_1 + U_2 + U_3)$, where each $U_j,~j = 1, 2, 3$, is a unitary operator and $\alpha$ is a constant.
\\$(ii)$ If $K$ is onto, then it can be written as a linear combination of two unitary operators if and only if $K$ is
invertible.
\end{proposition}

Every  closed-ranged operator
has a \emph{right-inverse} operator in the following sense:

\begin{lemma}\label{pseudo}
{\rm(\cite{chris})} Let $H$ and $K$ be Hilbert spaces, and suppose that $U : K \longrightarrow H$ is a bounded operator with closed range $R(U)$. Then there exists a bounded
operator $U^{\dag}:H\longrightarrow K$ for which
$$UU^{\dag}h = h,\quad h\in R(U).$$
\end{lemma}
The operator $U^{\dag}$ is called the \emph{pseudo-inverse} of $U$.

%%%%%%%%%%%%%%%%%%%%%%%%%%%%%%%%%%%%%%%%%%%%%%%%%%%%%%%%%%%%%%%%%%%%%%%%%%%%%%%%%%%%%%%%%%%%%%%%%%%%%%%%%%%%%%%%%%%%%%%%%%%%%%%%%%%%%%%%%%%%%%%%%%%%%%%%%%%%%%%%%%%%
\section{$cg$-Orthonormal bases}

Similar to the continuous frames, we want to generalize orthonormal bases. Indeed, our purpose here is to define  a mapping $f:\Omega\longrightarrow H$  that has  similar properties to an orthonormal basis of $H$.
\begin{definition}
Suppose $(\Omega,\mu)$ is a measure space. A mapping $f:\Omega\longrightarrow H$ is called  a \emph{$c$-orthonormal basis}  with respect to $(\Omega,\mu)$ for $H$, if:
\\$(i)$ For each $h\in H$, $\omega\longmapsto \langle h,f(\omega)\rangle$ is measurable,
\\$(ii)$ for almost all $\nu\in \Omega$, $$\int_{\Omega}\langle f(\omega),f(\nu)\rangle d\mu(\omega)=1,$$
\\$(ii)$ for each $h\in H$, $\int_{\Omega}|\langle h,f(\omega)\rangle|^{2}d\mu(\omega)=\|h\|^{2}$.
\end{definition}

Now we define generalization of orthonormal basis in case of operators.

\begin{definition}
Assume $(\Omega,\mu)$ is a measure space. A family of operators $\Lambda=\{\Lambda_{\omega}\in B(H,H_\omega):  \omega \in\Omega\}$ is called a \emph{continuous $g$-orthonormal basis} or simply a \emph{ $cg$-orthonormal basis},  for $H$ with respect to $\{H_\omega\}_{\omega \in\Omega}$, whenever:
\\$(i)$ For each $h\in H$, $\{\Lambda_{\omega}h\}_{\omega \in\Omega}$ is strongly measurable,
\\$(ii)$ for almost all $\nu\in \Omega$,
$$\int_{\Omega}\langle \Lambda^{*}_{\omega}f_{\omega},\Lambda^{*}_{\nu}g_{\nu}\rangle d\mu(\omega)=\langle f_{\nu},g_{\nu}\rangle, ~\,\{f_{\omega}\}_{\omega \in\Omega} \in (\oplus_{\omega \in\Omega}H_\omega,\mu)_{L^{2}}, ~g_{\nu}\in H_\nu,$$
\\$(iii)$ for each $h\in H$, $\int_{\Omega}\|\Lambda_{\omega}h\|^{2} d\mu(\omega)=\|h\|^{2}$.

\vspace{2mm}
If only conditions $(i)$ and $(ii)$ hold,  $\Lambda=\{\Lambda_{\omega}\in B(H,H_\omega); \omega \in\Omega\}_{\omega \in\Omega}$ is called a \emph{$cg$-orthonormal system}  for $H$ with respect to $\{H_\omega\}_{\omega \in\Omega}$.
\end{definition}

\begin{example}
Suppose that $\Omega=\{a,b,c\}$, $\Sigma=\big\{\emptyset, \{a,b\},\{c\},\Omega\big\}$ and $\mu:\Sigma\longrightarrow [0,\infty]$ is a measure such that $\mu(\emptyset)=0$, $\mu(\{a,b\})=1$,  $\mu(\{c\})=1$ and $\mu(\Omega)=2$. Let $H$ be a $2$ dimensional Hilbert space with an orthonormal basis $\{e_1,e_2\}$. We define
$$f:\Omega\longrightarrow H$$
by
 $f=e_1\chi_{\{a,b\}}+ e_2\chi_{\{c\}}$.
So for each $h\in H$,
 $$\langle f(\omega),h\rangle=\langle e_1,h\rangle\chi_{\{a,b\}}(\omega)+\langle e_2,h\rangle\chi_{\{c\}}(\omega), \quad \omega\in \Omega,$$
hence $\omega\longmapsto \langle h,f(\omega)\rangle$ is measurable.
Now, for each $\omega\in \Omega$, we define
$$\Lambda_\omega:H \longrightarrow \mathbb{C}$$
$$\Lambda_\omega(h)=\langle h,f(\omega)\rangle.$$
Actually, we consider for each  $\omega\in \Omega$, $H_\omega=\mathbb{C}$. By an easy calculation, we have
$$\Lambda_\omega^*(z)=f(\omega)z, ~~ z \in  \mathbb{C}.$$
For any $\nu\in \Omega$, $x_{\nu} \in \mathbb{C}$  and any $\{z_{\omega}\}_{\omega \in\Omega} \in (\oplus_{\omega \in\Omega}H_\omega,\mu)_{L^{2}}=L^2(\Omega,\mu)$, due to the Example 4.2 in \cite{Ra}, we have
$$\int_{\Omega}\langle \Lambda^{*}_{\omega}z_{\omega},\Lambda^{*}_{\nu}x_{\nu}\rangle d\mu(\omega)=\int_{\Omega} z_\omega \overline{x_\nu} \langle f(\omega),f(\nu)\rangle d\mu(\omega)=z_\nu \overline{x_\nu}.$$
Also  for each $h\in H$,
$$\int_{\Omega}\|\Lambda_{\omega}h\|^{2} d\mu(\omega)=\int_{\Omega}|\langle h,f(\omega)\rangle|^{2}d\mu(\omega)=\|h\|^{2}.$$
Therefore $\{\Lambda_\omega\}_{\omega \in \Omega}$ is a $cg$-orthonormal basis for $H$ with respect to $\{H_\omega\}_{\omega \in\Omega}$, where for each  $\omega\in \Omega$, $H_\omega=\mathbb{C}$.
\end{example}
 We present some equal conditions for  $cg$-orthonormal bases.
\begin{theorem}
Let $\{\Lambda_\omega\}_{\omega \in \Omega}$ be a $cg$-orthonormal system for $H$ with respect to $\{H_\omega\}_{\omega \in\Omega}$. Also assume that for each $h\in H$,  $\int_{\Omega}\|\Lambda_{\omega}h\|^{2} d\mu(\omega)< \infty$.
Then the following conditions are equivalent:
\\$(i)$ $\{\Lambda_\omega\}_{\omega \in \Omega}$ is a $cg$-orthonormal basis for $H$ with respect to $\{H_\omega\}_{\omega \in\Omega}$.
\\$(ii)$ For each $h, k \in H$,
$$\langle h,k \rangle=\int_{\Omega} \langle \Lambda_\omega h, \Lambda_\omega k \rangle d\mu(\omega).$$
\\$(iii)$ If  $\Lambda_\omega h=0$, $a.e.~ [\mu]$, then $h=0$.
\\$(\upsilon)$ For each zero measure  set $\Omega_0\subseteq \Omega$, $H=\overline{span}\{\Lambda_\omega^*(H_\omega)\}_{\omega \in \Omega \backslash \Omega_0}$.
\end{theorem}
\begin{proof}
$(i)  \Longleftrightarrow  (ii)$ Since $\{\Lambda_\omega\}_{\omega \in \Omega}$ is a Parseval $cg$-frame for $H$, so its frame operator $S_\Lambda=I$. Hence $(ii)$  is obvious. The converse side clearly holds.
\\ $(ii)  \Rightarrow  (iii)$  If  $\Lambda_\omega h=0$, $a.e.~ [\mu]$, then for every $k \in H$,
$$\langle h,k \rangle= \int_{\Omega} \langle \Lambda_\omega h, \Lambda_\omega k \rangle d\mu(\omega)=0.$$
Therefore $h=0$.
\\ $(iii)  \Rightarrow  (\upsilon)$ Suppose that  $\Omega_0\subseteq \Omega$ and $h\perp\overline{span}\{\Lambda_\omega^*(H_\omega)\}_{\omega \in \Omega \backslash \Omega_0}$, so for almost all $\omega\in \Omega$, $\langle \Lambda_\omega^* \Lambda_\omega h, h \rangle=0$. Then   $\|\Lambda_\omega h\|^2=0$, $a.e.~ [\mu]$, which implies $h=0$.
\\ $(\upsilon)  \Rightarrow  (ii)$ Let $k\in H$. Assume that
$$\mathcal{H}_k=\Big\{h\in H : \langle h,k\rangle= \int_{\Omega}\langle\Lambda_\omega h,\Lambda_\omega k\rangle d\mu(\omega)\Big\}.$$
$\mathcal{H}_k$ is a subspace of $H$. Also, it is closed, since if $\lim_{n\rightarrow \infty}h_n=h$, where $h_n$'s belong to $\mathcal{H}_k$, then
$$\langle h,k \rangle=\lim_{n\rightarrow \infty} \langle h_n,k \rangle=\lim_{n\rightarrow \infty} \int_{\Omega}\langle\Lambda_\omega h_n,\Lambda_\omega k\rangle d\mu(\omega). $$
According to assumption,
\begin{align*}
\int_{\Omega}|\langle\Lambda_\omega h,\Lambda_\omega k\rangle| d\mu(\omega)\leq \Big(\int_{\Omega}\|\Lambda_\omega h\|^2 d\mu(\omega)\Big)^{\frac{1}{2}}\Big(\int_{\Omega}\|\Lambda_\omega g\|^2 d\mu(\omega)\Big)^{\frac{1}{2}}<\infty.
\end{align*}
So by Lebesgue's Dominated Convergence Theorem,
$$\lim_{n\rightarrow \infty} \int_{\Omega}\langle\Lambda_\omega h_n,\Lambda_\omega k\rangle d\mu(\omega)=
 \int_{\Omega} \langle\Lambda_\omega h,\Lambda_\omega k\rangle d\mu(\omega),$$
which means  $h\in \mathcal{H}_k$.
For almost all $\nu \in \Omega$ and each $f\in H$, we have
\begin{align*}
\int_{\Omega}\langle\Lambda_\omega \Lambda^*_\nu \Lambda_\nu f,\Lambda_\omega k \rangle d\mu(\omega)
&=\int_{\Omega}\langle \Lambda^*_\nu \Lambda_\nu f,\Lambda_\omega^* \Lambda_\omega k \rangle d\mu(\omega)
\\&=\langle \Lambda_\nu f, \Lambda_\nu k\rangle
=\langle \Lambda_\nu^*\Lambda_\nu f,k\rangle.
\end{align*}
Therefore  $\Lambda^*_\nu \Lambda_\nu f \in \mathcal{H}_k$.
Assume  $f\perp \mathcal{H}_k$, then for almost all $\nu \in \Omega$,
$$0=\langle \Lambda^*_\nu \Lambda_\nu f, f\rangle=\|\Lambda_\nu f\|^2,$$
which gives $\Lambda_\nu f=0$.  For almost all $\nu \in \Omega$ and any $g_\nu \in H_\nu$,
$$\langle f, \Lambda^*_\nu g_\nu\rangle=\langle \Lambda_\nu f, g_\nu \rangle=0.$$
So $f\perp\overline{span}\{\Lambda_\omega^*(H_\omega)\}_{\omega \in \Omega \backslash \Omega_0}$, where $\Omega_0$
is a zero measure subset of $\Omega$. By assumption  $(\upsilon)$, $\overline{span}\{\Lambda_\omega^*(H_\omega)\}_{\omega \in \Omega \backslash \Omega_0}=H$. Thus $f=0$ and $\mathcal{H}_k=H$. Therefore
$$\langle h,k \rangle=\int_{\Omega} \langle \Lambda_\omega h, \Lambda_\omega k \rangle d\mu(\omega),\quad h,k \in H.$$
\end{proof}
In the following,  suppose that there exists a $cg$-orthonormal basis for $H$.

\begin{proposition}\label{t1}
Suppose that $\{\Theta_\omega\}_{\omega \in \Omega}$ is a $cg$-orthonormal basis for $H$ with respect to $\{H_\omega\}_{\omega \in\Omega}$ and $\{\Lambda_{\omega}\in B(H,H_\omega):  \omega \in\Omega\}$ is a family such that for each $h\in H$, $\{\Lambda_\omega h\}_{\omega \in \Omega}$ is strongly measurable. Then $\{\Lambda_\omega\}_{\omega \in \Omega}$ is a Parseval $cg$-frame for $H$ if and only if there exists a unique isometry $V \in B(H)$ such that $\Lambda_\omega=\Theta_\omega V$, $a.e.~ [\mu]$.
\end{proposition}

\begin{proof}
Let $\{\Lambda_\omega\}_{\omega \in \Omega}$ is a Parseval $cg$-frame  for $H$. We define the operator $V$ weakly by
$$\langle Vf,h\rangle=\int_{\Omega}\langle\Theta^{*}_{\omega}\Lambda_{\omega}f,h\rangle d\mu(\omega),\quad f,h \in H.$$
For each $f,h\in H$, we have
\begin{align*}
|\langle Vf,h\rangle|\leq \Big(\int_{\Omega}\|\Lambda_\omega f\|^2d\mu(\omega)\Big)^{\frac{1}{2}} \Big(\int_{\Omega}\|\Theta_\omega h\|^2d\mu(\omega)\Big)^{\frac{1}{2}}
\leq \|f\| \|h\|.
\end{align*}
So $V$ is well-defined and bounded.
%Let $T_\Lambda$ and $T_\Theta$ be the synthesis operators of $\{\Lambda_\omega\}_{\omega \in \Omega}$ and $\{\Theta_\omega\}_{\omega \in \Omega}$, respectively.
For almost all $\nu \in \Omega$ and each $f \in H$, $h_\nu\in H_\nu$,
\begin{align*}
\langle \Theta_\nu Vf, h\rangle=\langle Vf,\Theta^{*}_{\nu}h\rangle=\int_{\Omega}\langle \Theta^*_\omega\Lambda_\omega f,\Theta^*_\nu h\rangle d\mu(\omega)=\langle\Lambda_\nu f, h\rangle.
\end{align*}
Since $\{\Theta_\omega\}_{\omega \in \Omega}$ is a $cg$-orthonormal basis, thus $\Lambda_\omega=\Theta_\omega V$, $a.e.~ [\mu]$.
\\For each $f\in H$,
\begin{align*}
\|Vf\|^2=\langle Vf, Vf\rangle&=\int_{\Omega}\langle \Theta^*_\omega\Lambda_\omega f ,Vf\rangle d\mu(\omega)=\int_{\Omega}\langle \Lambda_\omega f,\Theta_{\omega}Vf\rangle d\mu(\omega)
\\&=\int_{\Omega}\langle \Lambda_\omega f ,\Lambda_\omega f\rangle d\mu(\omega)=\|f\|^{2}.
\end{align*}
Therefore $V$ is an isometry.

Now, let $V_1$ and $V_2$ be two isometries such that $\Lambda_\omega=\Theta_\omega V_1$, $a.e.~ [\mu]$ and $\Lambda_\omega=\Theta_\omega V_2$, $a.e.~ [\mu]$.
Then for each $f \in H$, $\Theta_\omega((V_1-V_2)f)=0$, $a.e.~ [\mu]$, which implies
$$0=\int_{\Omega}\|\Theta_\omega(V_1-V_2)f\|^2d\mu(\omega)=\|(V_1-V_2)f\|^2,$$
so $V_1f=V_2f$.

Conversely, let $V \in B(H)$ be a unique isometry such that $\Lambda_\omega=\Theta_\omega V$, $a.e.~ [\mu]$.
For any $f\in H$,
$$\int_{\Omega}\|\Lambda_{\omega}f\|^{2} d\mu(\omega)=\int_{\Omega}\|\Theta_{\omega}Vf\|^{2} d\mu(\omega)=\|Vf\|^{2}=\|f\|^{2}.$$
Hence $\{\Lambda_\omega\}_{\omega \in \Omega}$ is a Parseval $cg$-frame for $H$.
\end{proof}

\begin{theorem}\label{ob1}
Assume  $\{\Theta_\omega\}_{\omega \in \Omega}$ is a $cg$-orthonormal basis for $H$ with respect to $\{H_\omega\}_{\omega \in\Omega}$ and $\{\Lambda_{\omega}\in B(H,H_\omega):  \omega \in\Omega\}$ is a family such that for each $h\in H$, $\{\Lambda_\omega h\}_{\omega \in \Omega}$ is strongly measurable. Then $\{\Lambda_\omega\}_{\omega \in \Omega}$ is a $cg$-frame for $H$ with bounds $A$ and $B$ if and only if there exists a unique $V \in B(H)$ such that $\Lambda_\omega=\Theta_\omega V$, $a.e.~ [\mu]$ and $AI\leq V^*V\leq BI$.
\end{theorem}
\begin{proof}
Let $\{\Lambda_\omega\}_{\omega \in \Omega}$ be a $cg$-frame for $H$. Similar to the proof of Proposition \ref{t1}, the operator $V$ weakly defined by
$$\langle Vf,h\rangle=\int_{\Omega}\langle\Theta^{*}_{\omega}\Lambda_{\omega}f,h\rangle d\mu(\omega),\quad f,h \in H,$$
is a one-to-one and bounded operator such that $\Lambda_\omega=\Theta_\omega V$, $a.e.~ [\mu]$.
Also for each $f\in H$,
\begin{align*}
\|Vf\|^2=\langle Vf, Vf\rangle&=\int_{\Omega}\langle \Theta^*_\omega\Lambda_\omega,Vf\rangle d\mu(\omega)=\int_{\Omega}\langle \Lambda_\omega f,\Theta_{\omega}Vf\rangle d\mu(\omega)
\\&=\int_{\Omega} \|\Lambda_\omega f \|^{2} d\mu(\omega).
\end{align*}
Therefore
 $$ A\langle f,f\rangle \leq \langle V^*Vf,f\rangle\leq B\langle f,f\rangle,$$
which implies $AI\leq VV^*\leq BI$.
\\The opposite implication is similar to the Proposition \ref{t1}.
\end{proof}

\begin{theorem}\label{ob2}
Suppose that $\{\Lambda_\omega\}_{\omega \in \Omega}$ is a $cg$-orthonormal basis for $H$ with respect to $\{H_\omega\}_{\omega \in\Omega}$  and $V \in B(H)$. Then $\{\Lambda_\omega V\}_{\omega \in \Omega}$ is a $cg$-orthonormal basis for $H$ with respect to $\{H_\omega\}_{\omega \in\Omega}$  if  and only if $V$ is unitary.
\end{theorem}
\begin{proof}
Assume that $\{\Lambda_\omega V\}_{\omega \in \Omega}$ is a $cg$-orthonormal basis for $H$. Since $\{\Lambda_\omega\}_{\omega \in \Omega}$ is also a $cg$-orthonormal basis for $H$, for each $f \in H$, we have
$$\|Vf\|^2=\int_{\Omega} \|\Lambda_\omega Vf \|^{2} d\mu(\omega)=\|f\|^2.$$
Hence $V$ is an isometry and $V^*V=I$. Considering $\Theta_\omega=\Lambda_\omega V, ~\omega\in \Omega$, in Theorem \ref{t1}, there exists a unique isometry $U\in B(H)$ such that $\Lambda_\omega=\Lambda_\omega VU$, $a.e.~ [\mu]$. Let $T_\Lambda$ and $T_{\Lambda VU}$ be
the synthesis operators of Parseval $cg$-frames $\{\Lambda_\omega\}_{\omega \in \Omega}$ and $\{\Lambda_\omega VU\}_{\omega \in \Omega}$, respectively.
Then $T_\Lambda=(VU)^*T_\Lambda=U^*V^*T_\Lambda$. We deduce $T_\Lambda T_\Lambda^*=U^*V^*T_\Lambda T_\Lambda^*$ or $S_\Lambda=U^*V^*S_\Lambda$, where $S_\Lambda$ is the frame operator of $\{\Lambda_\omega \}_{\omega \in \Omega}$. Since $S_\Lambda=I$, so $I=U^*V^*$ or equivalently $VU=I$.
This implies that $V$ is onto. Also $V$ is one-to-one, so $V$ is invertible and $V^{-1}=V^{*}$, which means  $V$ is a unitary.

For the other implication, suppose that $V$ is a unitary operator. Now, we show that $\{\Lambda_\omega V\}_{\omega \in \Omega}$ is a $cg$-orthonormal basis for $H$.
For almost all $\nu \in \Omega$, each $g_\nu \in H_\nu$ and each $\{f_{\omega}\}_{\omega \in\Omega} \in (\oplus_{\omega \in\Omega}H_\omega,\mu)_{L^{2}}$, we have
\begin{align*}
\int_{\Omega}\langle (\Lambda_\omega V)^*f_\omega,(\Lambda_\nu V)^*g_\nu\rangle d\mu(\omega)
&=\int_{\Omega}\langle V^*\Lambda_\omega^*f_\omega,V^*\Lambda_\nu^*g_\nu\rangle d\mu(\omega)
\\&=\int_{\Omega}\langle \Lambda_\omega^*f_\omega,\Lambda_\nu^*g_\nu\rangle d\mu(\omega)=\langle f_\nu, g_\nu\rangle.
\end{align*}
Also for each $f\in H$,
$$\int_{\Omega}\|\Lambda_\omega Vf\|^2d\mu(\omega)=\|Vf\|^2=\|f\|^2.$$
Therefore $\{\Lambda_\omega V\}_{\omega \in \Omega}$ is a $cg$-orthonormal basis for $H$ with respect to $\{H_\omega\}_{\omega \in\Omega}$.
\end{proof}

Concerning to $cg$-Riesz bases which are defined in \cite{mad}, we have next result.

\begin{theorem}
Let $(\Omega,\mu)$ be a measure space where $\mu$ is $\sigma$-finite.
Suppose that  $\{\Lambda_\omega\}_{\omega \in \Omega}$ is a $cg$-Riesz basis for $H$ and $V \in B(H)$. Then $\{\Lambda_\omega V\}_{\omega \in \Omega}$ is a $cg$-Riesz basis for $H$ if  and only if $V$ is invertible.
\end{theorem}
\begin{proof}
Let $\{\Lambda_\omega V\}_{\omega \in \Omega}$ be a $cg$-Riesz basis  for $H$. By definition of a $cg$-Riesz basis, the operator $T_{\Lambda V}$ weakly defined by
$$\langle T_{\Lambda V}\varphi,h\rangle=\int_{\Omega}\langle (\Lambda_\omega V)^*\varphi(\omega),h\rangle d\mu(\omega),\quad \varphi  \in (\oplus_{\omega \in\Omega}H_\omega,\mu)_{L^{2}},h\in H,$$
is well-defined and there exist positive constants $A$ and $B$ such that
$$A\|\varphi\|\leq \|T_{\Lambda V}\varphi\|\leq B \|\varphi\|, \quad \varphi  \in (\oplus_{\omega \in\Omega}H_\omega,\mu)_{L^{2}}.$$
An easy calculation shows $T_{\Lambda V}=V^*T_\Lambda$, where $T_\Lambda$ is defined similarly for  $\{\Lambda_\omega\}_{\omega \in \Omega}$. By Lemma 3.2 $(i)$  in \cite{mad}, $T_{\Lambda V}$ and $T_\Lambda$ both are invertible. So $V^*=T_{\Lambda V}T_\Lambda
^{-1}$ is invertible and $V$ is invertible.
\\Conversely, let $V \in B(H)$ be invertible. If $\Lambda_\omega Vf=0$, $a.e.~ [\mu]$, then $Vf=0$ and it implies $f=0$. The operator $T_{\Lambda V}$ given by
$$\langle T_{\Lambda V}\varphi,h\rangle=\int_{\Omega}\langle (\Lambda_\omega V)^*\varphi(\omega),h\rangle d\mu(\omega),\quad \varphi  \in (\oplus_{\omega \in\Omega}H_\omega,\mu)_{L^{2}},h\in H,$$
is well-defined, bounded and $T_{\Lambda V}=V^*T_\Lambda$, where $T_\Lambda$ is defined similar to $T_{\Lambda V}$. Also, there are positive constants
 $A$ and $B$  such that
 $$A\|\varphi\|\leq \|T_{\Lambda }\varphi\|\leq B \|\varphi\|, \quad \varphi  \in (\oplus_{\omega \in\Omega}H_\omega,\mu)_{L^{2}}.$$
 For each $\varphi  \in (\oplus_{\omega \in\Omega}H_\omega,\mu)_{L^{2}}$, $ \|T_{\Lambda V }\varphi\|=\|V^*T_{\Lambda }\varphi\|\leq B \|V^*\|\|\varphi\|$, and
 $$\|T_{\Lambda V }\varphi\|=\|V^*T_{\Lambda }\varphi\|\geq \frac{1}{\|V^{-1}\|}\|T_{\Lambda }\varphi\|\geq \frac{A}{\|V^{-1}\|}\|\varphi\|.$$
 This shows that $\{\Lambda_\omega V\}_{\omega \in \Omega}$ is a $cg$-Riesz basis for $H$.
\end{proof}

Now, we define  $cg$-complete families  as follows:

\begin{definition}
A family $\{\Lambda_{\omega}\in B(H,H_{\omega}): \omega \in \Omega\}$ is called a $cg$-complete family  for $H$ with respect to $\{H_\omega\}_{\omega\in \Omega}$, if:
 $$\big\{h:~\Lambda_{\omega}h=0,\,a.e.~[\mu]\big\}=\{0\}.$$
\end{definition}

\begin{lemma}\label{complete}
$\{\Lambda_{\omega}\}_{\omega \in\Omega}$ is a $cg$-complete family  for $H$ and $V \in B(H)$. Then $\{\Lambda_{\omega}V\}_{\omega \in\Omega}$ is a $cg$-complete family  for $H$ if and only if $V$ is one-to-one.
\end{lemma}
\begin{proof}
Let $\{\Lambda_{\omega}V\}_{\omega \in\Omega}$ be a $cg$-complete family. If $Vh=0$, then $$\Lambda_{\omega}Vh=0, ~\omega \in\Omega,$$  so $h=0$ and $V$ is one-to-one.
\\ Now, suppose  $V$ is one-to-one and $\Lambda_{\omega}Vh=0, ~a.e.~ [\mu]$. Since $\{\Lambda_{\omega}\}_{\omega \in\Omega}$ is $cg$-complete, $Vh=0$. Hence $h=0$ , which implies $\{\Lambda_{\omega}V\}_{\omega \in\Omega}$ is  $cg$-complete.
\end{proof}

\begin{proposition}
Let $(\Omega,\mu)$ be a measure space where $\mu$ is $\sigma$-finite
and $\{\Lambda_\omega\}_{\omega \in \Omega}$ be a $cg$-Bessel family  for $H$. Then $\{\Lambda_\omega\}_{\omega \in \Omega}$ is $cg$-complete if and only if
$\overline{R(T_\Lambda)}=H$, where $T_\Lambda$ is the synthesis operator of $\{\Lambda_\omega\}_{\omega \in \Omega}$.
\end{proposition}

\begin{proof}
Assume that $\{\Lambda_\omega\}_{\omega \in \Omega}$ is $cg$-complete. To show that $\overline{R(T_\Lambda)}=H$, it is enough to prove that if $f\in H$ and
$f\bot R(T_\Lambda)$, so $f=0$. Let $f\in H$ and $f\bot R(T_\Lambda)$, then for each $F \in  (\oplus_{\omega \in\Omega}H_\omega,\mu)_{L^{2}}$,
$$0=\langle T_\Lambda F, f\rangle=\int_{\Omega} \langle \Lambda^*_\omega F(\omega),f\rangle d\mu(\omega).$$
Since $(\Omega,\mu)$ is $\sigma$-finite, there exists a family $\{\Omega_n\}_{n=1}^{\infty}$ of disjoint measurable subsets of $\Omega$,  such that
$\Omega=\bigcup_{n=1}^{\infty}\Omega_n$ and $\mu(\Omega_n)<\infty$, $n\geq 1$. For each $n\geq 1$, set
$$F_n(\omega)=\left\{
     \begin{array}{ll}
      \Lambda_\omega f,\quad &  \hbox{$\omega\in \Omega_n$} \\
      0,\quad &  \hbox{$otherwise$}
     \end{array}
   \right.,$$
then
$$\langle T_\Lambda F_n,f\rangle=\int_{\Omega} \langle  F(\omega),\Lambda_\omega f\rangle d\mu(\omega)=\|\Lambda_\omega f\|^2\mu(\Omega_n)=0.$$
Thus  $\Lambda_\omega f=0$, $a.e.~ [\mu]$, which implies $f=0$. So $\overline{R(T_\Lambda)}=H$.
\\For the opposite implication, suppose that $\overline{R(T_\Lambda)}=H$ and there exists a $f\neq 0$ such that
$$\Lambda_\omega f=0, ~a.e.~ [\mu].$$
 There exists a sequence
$\{F_n\}_{n=1}^{\infty} \in (\oplus_{\omega \in\Omega}H_\omega,\mu)_{L^{2}}$ such that $\lim_{n\rightarrow \infty} T_\Lambda F_n=f$. Then
\begin{align*}
\|f\|^{2}&=\langle f,f\rangle=\langle \lim_{n\rightarrow \infty} T_\Lambda F_n, f\rangle
=\lim_{n \rightarrow \infty} \langle T_\Lambda F_n, f\rangle
\\&=\lim_{n \rightarrow \infty} \int_{\Omega}\langle \Lambda_\omega^*F_n(\omega),f\rangle d\mu(\omega)
\\&=\lim_{n \rightarrow \infty} \int_{\Omega}\langle F_n(\omega),\Lambda_\omega f\rangle d\mu(\omega)=0,
\end{align*}
which is a contradiction.
\end{proof}

\begin{remark}
 Let $(\Omega,\mu)$ be a measure space and consider the family $\{\Lambda_{\omega}\in B(H,H_{\omega}): \omega \in \Omega\}$.
%$\{\Lambda_\omega\}_{\omega \in \Omega}$ be a  $cg$-frame for $H$ with respect to  $\{H_\omega\}_{\omega \in \Omega}$.
Also, suppose that $\{e_{\omega,k}\}_{\omega \in \Omega, k\in \mathbb{K}_\omega }$ is an orthonormal basis for Hilbert space
$\oplus_{\omega \in\Omega}H_\omega$  such that for each $\omega \in \Omega$,  $\{e_{\omega,k}\}_{ k\in \mathbb{K}_\omega }$ is an orthonormal basis of $H_\omega$ and for each $h \in H$, the mapping
\begin{align*}
&\Omega\times\mathbb{K}\longrightarrow\mathbb{C}
\\(\omega&,k)\longmapsto \langle h,e_{\omega,k}\rangle
\end{align*}
is measurable,  where $\mathbb{K}=\bigcup _{\omega \in \Omega} \mathbb{K}_\omega$.

The mapping $h\longmapsto \langle\Lambda_\omega h,e_{\omega,k}\rangle$, $\omega \in \Omega,~k\in \mathbb{K}_\omega$, defines a bounded linear functional on $H$. So there exist some $u_{\omega,k}\in H$ such that
$$\langle h,u_{\omega,k}\rangle=\langle \Lambda_\omega h,e_{\omega,k}\rangle, \quad h\in H, ~ \omega \in \Omega,~k\in \mathbb{K}_\omega.$$
Therefore $\Lambda_\omega h=\sum_{k\in \mathbb{K}_\omega}\langle h,u_{\omega,k}\rangle e_{\omega,k}$, $h\in H$.
Since $$\sum_{k\in \mathbb{K}_\omega}|\langle h,u_{\omega,k}\rangle e_{\omega,k}|^2=\|\Lambda_\omega h\|^2\leq \|\Lambda_\omega\|^2\|h\|^2,$$
so for each $\omega \in \Omega$, $\{u_{\omega,k}\}_{k\in \mathbb{K}_\omega}$ is a $c$-Bessel family for $H_\omega$.
Also
\begin{align}\label{u}
u_{\omega,k}=\Lambda^*e_{\omega,k}, \quad \omega \in \Omega,~k\in \mathbb{K}_\omega.
\end{align}
The family $\{u_{\omega,k}\}_{\omega \in \Omega,k\in \mathbb{K}_\omega}$  is called the family induced by $\{\Lambda_\omega\}_{\omega \in \Omega}$ with respect to  $\{e_{\omega,k}\}_{\omega \in \Omega, k\in \mathbb{K}_\omega }$.

Consider the mapping $u(\omega,k):\Omega\times\mathbb{K}\longrightarrow H$ defined by
$$u(\omega,k)=\left\{
     \begin{array}{ll}
      u_{\omega,k}~,\quad &  \hbox{$k\in \mathbb{K}_\omega$} \\
      0~,\quad &  \hbox{$otherwise$}
     \end{array}
   \right.,$$
 where $\mathbb{K}=\bigcup _{\omega \in \Omega} \mathbb{K}_\omega$.

\vspace{1mm}
For each $h\in H$, $(\omega,k)\longmapsto \langle u_{\omega,k},h\rangle$ is measurable and
\begin{align}\label{uu}
 \int_{\Omega} \|\Lambda_\omega h\|^2 d\mu(\omega)&= \int_{\Omega}\sum_{k\in \mathbb{K}_\omega} |\langle h,  u_{\omega,k}\rangle|^2d\mu(\omega) \nonumber
 \\&= \int_{\Omega} \big(\int_{\mathbb{K}}|\langle h,  u_{\omega,k}\rangle|^2 dl(k)\big) d\mu(\omega),
\end{align}
where $l:\mathbb{K}\longrightarrow \mathbb{K}$ is the counting measure on $\mathbb{K}$.
If  $\{\Lambda_\omega\}_{\omega \in \Omega}$  is a $cg$-frame for $H$ with respect to $\{H_\omega\}_{\omega \in\Omega}$, then  $u$ is  a $c$-frame for $H$ with respect to ($\Omega\times\mathbb{K}, \mu \times\ l)$ and  with the same bounds of $\{\Lambda_\omega\}_{\omega \in \Omega}$.

The converse of that is true, too; if  $\{u_{\omega,k}\}_{\omega \in \Omega,k\in \mathbb{K}_\omega}$ is a $c$-frame for $H$, then $\{\Lambda_\omega\}_{\omega \in \Omega}$ is a $cg$-frame for $H$ with the same bounds of $\{u_{\omega,k}\}_{k\in \mathbb{K}_\omega}$.
\end{remark}

\begin{theorem}
Let $(\Omega,\mu)$ be a measure space where $\mu$ is $\sigma$-finite. Consider the family $\{\Lambda_{\omega}\in B(H,H_\omega);\omega \in \Omega\}$ and let $\{u_{\omega,k}\}_{\omega \in \Omega,k\in \mathbb{K}_\omega}$ be defined as in (\ref{u}). Then $\{\Lambda_\omega\}_{\omega \in \Omega}$ is a $cg$-frame (respectively, $cg$-Bessel family, tight $cg$-frame,
$cg$-Riesz basis, $cg$-orthonormal basis) for $H$ if and only if $\{u_{\omega,k}\}_{\omega \in \Omega,k\in \mathbb{K}_\omega}$ is a $c$-frame (respectively, $c$-Bessel family, tight $c$-frame, $c$-Riesz basis, $c$-orthonormal basis) for $H$.
\end{theorem}
\begin{proof}
We see from (\ref{uu}) that
$$\int_{\Omega} \|\Lambda_\omega h\|^2 d\mu(\omega)=\int_{\Omega} \big(\int_{\mathbb{K}}|\langle h,u_{\omega,k}\rangle|^2 dl(k)\big) d\mu(\omega),\quad h\in H.$$
Hence  $\{\Lambda_\omega\}_{\omega \in \Omega}$ is a  $cg$-frame (respectively, $cg$-Bessel family, tight $cg$-frame) for $H$ if and only if $\{u_{\omega,k}\}_{\omega \in \Omega,k\in \mathbb{K}_\omega}$ is a $c$-frame (respectively, $c$-Bessel family, tight $c$-frame).
\\Now, assume that $\{\Lambda_\omega\}_{\omega \in \Omega}$ is a  $cg$-Riesz basis for $H$.
So there are constants $A,B>0$ such that the operator
$T_\Lambda: (\oplus_{\omega \in\Omega}H_\omega,\mu)_{L^{2}}\longrightarrow H$ defined by
$$\langle T_\Lambda F, h\rangle=\int_{\Omega} \langle \Lambda^*_\omega F(\omega),h\rangle d\mu(\omega),\quad F\in(\oplus_{\omega \in\Omega}H_\omega,\mu)_{L^{2}},h\in H,$$
satisfies in
$$A\|F\|\leq \|T_\Lambda F\|\leq B\|F\|,~ F\in (\oplus_{\omega \in\Omega}H_\omega,\mu)_{L^{2}}.$$
Consider the operator $\mathfrak{T}: L^2(\Omega\times\mathbb{K})\longrightarrow H$ which is defined by
\begin{align*}
\langle \mathfrak{T} \varphi, h\rangle=&\int_{\Omega} \int_{\mathbb{K}}\varphi(\omega,k) \langle  u_{\omega,k}, h\rangle dl(k) d\mu(\omega)
\\=&\int_{\Omega} \sum_{k\in \mathbb{K}_\omega} \varphi(\omega,k) \langle  u_{\omega,k}, h\rangle  d\mu(\omega)
,\quad \varphi\in L^2(\Omega\times\mathbb{K}),h\in H.
\end{align*}
To show that $\{u_{\omega,k}\}_{k\in \mathbb{K}_\omega}$ is a $c$-Riesz basis for $H$, it is enough to show that
 $\mathfrak{T}$ is one-to-one (by Theorem 2.1 in \cite{Ra}). If $\mathfrak{T} \varphi=0$, then for each $h\in H$,
\begin{align*}
0=\langle \mathfrak{T} \varphi, h\rangle%=&\int_{\Omega} \int_{\mathbb{K}}\varphi(\omega,k) \langle  u_{\omega,k}, h\rangle d\mu(l) d\mu(\omega)
=&\int_{\Omega} \sum_{k\in \mathbb{K}_\omega} \varphi(\omega,k) \langle  \Lambda^*_\omega e_{\omega,k}, h\rangle  d\mu(\omega)
\\=&\int_{\Omega} \langle\sum_{k\in \mathbb{K}_\omega} \varphi(\omega,k)  e_{\omega,k},  \Lambda_\omega h\rangle  d\mu(\omega)
\\=&\int_{\Omega} \langle \Lambda^*_\omega\psi,h\rangle  d\mu(\omega)=\langle T_\Lambda\psi,h\rangle=0,
\end{align*}
where $\psi(\omega)= \sum_{k\in \mathbb{K}_\omega} \varphi(\omega,k)  e_{\omega,k},~\omega\in \Omega$. So $T_\Lambda\psi=0$. Since $T_\Lambda$ is bounded below,
$\psi=0$. But $\|\psi\|=\|\varphi\|$, so $\varphi=0$. Hence $\mathfrak{T}$ is one-to-one and it implies $\{u_{\omega,k}\}_{k\in \mathbb{K}_\omega}$ is a $c$-Riesz basis for $H$.
\\Now, let $\{u_{\omega,k}\}_{\omega \in\Omega,k\in \mathbb{K}_\omega}$ be a  $c$-Riesz basis for $H$. By Theorem 3.3 in \cite{mad}, it suffices to show that $T_\Lambda$ is one-to-one. If $T_\Lambda\phi=0$, then for each $h\in H$,
\begin{align*}
0=\langle T_\Lambda\phi, h\rangle=&\int_{\Omega}\langle\Lambda^*_\omega\phi(\omega),h\rangle d\mu(\omega)%=&\int_{\Omega} \int_{\mathbb{K}}\varphi(\omega,k) \langle  u_{\omega,k}, h\rangle d\mu(l) d\mu(\omega)
=\int_{\Omega} \langle \Lambda^*_\omega(\sum_{k\in \mathbb{K}_\omega} \langle\phi(\omega),  e_{\omega,k}\rangle  e_{\omega,k}),h\rangle d\mu(\omega)
\\=&\int_{\Omega} \sum_{k\in \mathbb{K}_\omega} \langle\phi(\omega),  e_{\omega,k}\rangle  \langle \Lambda^*_\omega e_{\omega,k},h\rangle  d\mu(\omega)
=(*),
\end{align*}
set $\varphi(\omega,k)=\langle\phi(\omega),  e_{\omega,k}\rangle$, $\omega \in \Omega,k\in \mathbb{K}_\omega$, then
$$\int_{\Omega} \sum_{k\in \mathbb{K}_\omega} |\langle\phi(\omega),  e_{\omega,k}\rangle|^2    d\mu(\omega)=\int_{\Omega} \|\phi(\omega)\|^2 d\mu(\omega)=\|\phi\|^2.$$
So $\varphi\in L^2(\Omega,\mathbb{K})$ and $\|\varphi\|=\|\phi\|$. Also
$$(*)=\int_{\Omega}\int_{\mathbb{K}}\varphi(\omega,k)\langle \Lambda^*_\omega e_{\omega,k},h\rangle dl(k) d\mu(\omega).$$
So for each $h\in H$, $0=\langle \mathfrak{T} \varphi, h\rangle=0$, hence $\mathfrak{T} \varphi=0$.
Since $\{u_{\omega,k}\}_{\omega \in\Omega,k\in \mathbb{K}_\omega}$ is a  $c$-Riesz basis for $H$, so
$\mathfrak{T} $ is invertible, which implies $\varphi=0$ and $\phi=0$. Thus $T_\Lambda$ is one-to-one.
\\Now, suppose that $\{\Lambda_\omega\}_{\omega \in \Omega}$ is a $cg$-orthonormal basis for $H$. For almost all $\nu \in \Omega$ and all $m\in \mathbb{K}$,
\begin{align*}
\int_{\Omega}\int_{\mathbb{K}}\langle u_{\omega,k},u_{\nu,m}\rangle dl(k) d\mu(\omega)
&=\int_{\Omega}\int_{\mathbb{K}}\langle \Lambda^*_\omega e_{\omega,k},\Lambda^*_\nu e_{\nu,m}\rangle dl(k) d\mu(\omega)
\\&=\int_{\mathbb{K}}\int_{\Omega}\langle \Lambda^*_\omega e_{\omega,k},\Lambda^*_\nu e_{\nu,m}\rangle d\mu(\omega) dl(k)
\\&=\int_{\mathbb{K}}\langle e_{\nu,k},e_{\nu,m}\rangle dl(k) =1.
\end{align*}
Also, for each $h \in H$,
$$\int_{\Omega}\|\Lambda_\omega h\|^2 d\mu(\omega)=\int_{\Omega}\int_{\mathbb{K}}|\langle h,u_{\omega,k}\rangle|^2 dl(k) d\mu(\omega)=\|h\|^2.$$
So $\{u_{\omega,k}\}_{\omega \in\Omega,k\in \mathbb{K}_\omega}$ is a $c$-orthonormal basis.
The opposite implications are similar.
\end{proof}

%----------------------------------------------------------------------------------------------------------------------------------------------
%----------------------------------------------------------------------------------------------------------------------------------------------

%%%%%%%%%%%%%%%%%%%%%%%%%%%%%%%%%%%%%%%%%%%%%%%%%%%%%%%%%%%%%%%%%%%%%%%%%%%%%%%%%%%%%%%%%%%%%%%%%%%%%%%%%%%%%%%%%%%%%%%%%%%%%%%%%%%%%%%%%%%%%%%%%%%%%%%%%%%%%%%%%%%%%%%%%%%%%%%%%
\section{$cg$-Orthonormal bases and $cg$-frames}

At first, we present some result on $cg$-frames which are constructed by composing with operators.
\begin{proposition}
Let $\{\Lambda_\omega\}_{\omega \in \Omega}$ be a $cg$-frame for $H$ with bounds $A$ and $B$ and $V \in B(H)$. Then $\{\Lambda_\omega V\}_{\omega \in \Omega}$
is a $cg$-frame for $H$ if and only if there exists a positive constat $\alpha$ such that
$$\|Vf\|^2\geq \alpha \|f\|^2, \quad f \in H.$$
\end{proposition}
\begin{proof}
Suppose  $\{\Lambda_\omega V\}_{\omega \in \Omega}$ is a $cg$-frame for $H$ with bounds $C$ and $D$. For each $f \in H$,
$$C\|f\|^2\leq \int_{\Omega} \|\Lambda_\omega Vf\|^2 d\mu(\omega) \leq B \|Vf\|^2,$$
so $\|Vf\|^2\geq\frac{C}{B}\|f\|^2$. Set $\alpha=\frac{C}{B}$, then the proof is done.
\\Conversely, let $\alpha$ be such that $$\|Vf\|^2\geq \alpha \|f\|^2, \quad f \in H.$$ For each $f \in H$,
$$ \int_{\Omega} \|\Lambda_\omega Vf\|^2 d\mu(\omega) \leq B \|Vf\|^2\leq B \|V\|^2  \|f\|^2,$$
and
$$ \int_{\Omega} \|\Lambda_\omega Vf\|^2 d\mu(\omega) \geq A \|Vf\|^2\geq A\alpha  \|f\|^2.$$
Hence $\{\Lambda_\omega V\}_{\omega \in \Omega}$ is a $cg$-frame for $H$.
\end{proof}
\begin{corollary}\label{2-2}
Let $\{\Lambda_\omega\}_{\omega \in \Omega}$ be a $cg$-frame for $H$  and $V \in B(H)$. Then $\{\Lambda_\omega V\}_{\omega \in \Omega}$ is a $cg$-frame for $H$ if and only if $V^*$ is onto.
\end{corollary}
\begin{proof}
By Lemma 2.4.1 $(iii)$ in \cite{chris}, it is obvious.
\end{proof}
\begin{corollary}
Let $M$ be a close subspace of $H$ and $\{\Lambda_\omega\}_{\omega \in \Omega}$ be a $cg$-frame for $H$  and $V \in B(H,M)$. Then $\{\Lambda_\omega V^*\}_{\omega \in \Omega}$
is a $cg$-frame for $M$ if and only if there exists a positive constat $\alpha$ such that
$$\|V^*f\|^2\geq \alpha \|f\|^2, \quad f \in M.$$
\end{corollary}
\begin{corollary}
Let $\{\Lambda_\omega\}_{\omega \in \Omega}$ be a tight $cg$-frame for $H$  with frame bound $A$ and $V \in B(H)$. Then $\{\Lambda_\omega V\}_{\omega \in \Omega}$ is a tight $cg$-frame for $H$ with frame bound $\alpha$ if and only if
$$\|Vf\|^2= \frac{\alpha}{A} \|f\|^2, \quad f \in H.$$
\end{corollary}

\begin{proposition}\label{004}
Suppose $\{\Theta_\omega\}_{\omega \in \Omega}$ is a $cg$-orthonormal basis for $H$ with respect to $\{H_\omega\}_{\omega \in \Omega}$   and  $\{\Lambda_\omega\}_{\omega \in \Omega}$  is  a $cg$-frame for $H$ with respect to $\{H_\omega\}_{\omega \in \Omega}$. Then there exists a bounded and one-to-one operator $V$ on $H$ such that
$\Lambda_\omega=\Theta_\omega V$, $a.e.~ [\mu]$.
Furthermore, $V$ is invertible if $\{\Lambda_\omega\}_{\omega \in \Omega}$ is a $cg$-Riesz basis for $H$  and $V$ is unitary if
 $\{\Lambda_\omega\}_{\omega \in \Omega}$ is a $cg$-orthonormal  basis for $H$.
\end{proposition}
\begin{proof}
By the proof of Theorem \ref{ob1}, the first part is obvious.
If $\{\Lambda_\omega\}_{\omega \in \Omega}$ is a $cg$-Riesz basis for $H$, then by definition of $V$ in the proof of Theorem \ref{ob1}, $V=T_\Theta T^*_\Lambda$. Theorem 3.3 in \cite{mad} implies $T^*_\Lambda$ is onto, So $V$ is onto and consequently  $V$ is invertible.
If $\{\Lambda_\omega\}_{\omega \in \Omega}$ is a $cg$-orthonormal for $H$, then Theorem \ref{ob2} implies the result.
\end{proof}

\begin{proposition}\label{09}
Suppose that $\{\Theta_\omega\}_{\omega \in \Omega}$ is a $cg$-orthonormal basis  for $H$ with respect to $\{H_\omega\}_{\omega \in \Omega}$  and  $\{\Lambda_\omega\}_{\omega \in \Omega}$  is  a $cg$-frame  for $H$ with respect to $\{H_\omega\}_{\omega \in \Omega}$. Then there exist $cg$-orthonormal bases $\{\Psi_\omega\}_{\omega \in \Omega}$, $\{\Gamma_\omega\}_{\omega \in \Omega}$ and  $\{\Phi_\omega\}_{\omega \in \Omega}$  for $H$ with respect to $\{H_\omega\}_{\omega \in \Omega}$   and a constant $\alpha$  such that $$\Lambda_\omega=\alpha(\Psi_\omega+\Gamma_\omega+\Phi_\omega), ~ a.e.~ [\mu].$$
\end{proposition}

\begin{proof}
Due to Proposition \ref{004} and Proposition \ref{06}, we have an operator $V\in B(H)$ so that $V=\alpha(U_1+U_2+U_3)$, where each $U_{j},~j=1,2,3,$ is a unitary operator and $\alpha$ is a constant. Then
$\Lambda_\omega=\Theta_\omega V=\alpha (\Theta_\omega U_{1}+\Theta_\omega U_{2}+\Theta_\omega U_{3})$, $a.e.~ [\mu]$.
It is obvious that every $\{\Theta_\omega U_j\}_{\omega \in \Omega},~ j=1,2,3,$ is a $cg$-orthonormal basis for $H$.
Assuming $\Psi_\omega=\Theta_\omega U_1$, $a.e.~ [\mu]$, $\Gamma_\omega=\Theta_\omega U_2$, $a.e.~ [\mu]$ and
$\Phi_\omega=\Theta_\omega U_3$, $a.e.~ [\mu]$, the proof is completed.
\end{proof}

\begin{proposition}
Consider $\{\Theta_\omega \}_{\omega \in \Omega}$ as a $cg$-orthonormal basis for $H$ with respect to $\{H_\omega\}_{\omega \in \Omega}$. If $\{\Lambda_\omega\}_{\omega \in \Omega}$  is  a $cg$-Riesz basis   for $H$, then $\{\Lambda_\omega \}_{\omega \in \Omega}$ is sum of two $cg$-orthonormal bases for $H$ with respect to $\{H_\omega\}_{\omega \in \Omega}$.
\end{proposition}

\begin{proof}
Let $\{\Lambda_\omega\}_{\omega \in \Omega}$  is  a $cg$-Riesz basis for $H$.  Proposition \ref{004} implies that there exists an invertible  $V\in B(H)$ such that $\Lambda_\omega=\Theta_\omega V$, $a.e.~ [\mu]$. Via Proposition \ref{06}, $V$ can be written as $V=aU_1 +bU_2$, where $U_1$ and $U_2$ are unitary operators. The rest of proof is similar to the proof of Proposition \ref{09}.
\end{proof}

Composing of a $cg$-orthonormal basis and an isometry, gives us a Parseval $cg$-frame.

\begin{proposition}\label{004.4}
If $V\in B(H)$ is an isometry and $\{\Theta_\omega \}_{\omega \in \Omega}$ is a $cg$-orthonormal basis for $H$, then $\{\Theta_\omega \Lambda\}_{\omega \in \Omega}$ is a Parseval  $cg$-frame for $H$.
\end{proposition}
\begin{proof}
A straightforward calculation gives the proof.
\end{proof}

Every bounded operator $V$ on $H$  has a representation in the form $V = U|V|$ (called
the \emph{polar decomposition} of $V$), where $U$ is a  partial isometry, $|V|$ is
a positive operator defined by $|V|=\sqrt{V^*V} $ and $ker U = ker V$.
\\Also, every positive operator $P$ on H with $\|P\|\leq 1$ can be written in the form $P =\frac{ 1}{2}(W+W^*)$,
where $W = P + i\sqrt{1-P^2}$ is unitary.
\\

Next theorem shows that we can represent a $cg$-frame by some Parseval $cg$-frames.

\begin{theorem}
Suppose that $\{\Theta_\omega \}_{\omega \in \Omega}$ is a $cg$-orthonormal basis for $H$. Every $cg$-frame for $H$ can be written as a linear combination of two  Parseval  $cg$-frames.
\end{theorem}

\begin{proof}
By Proposition \ref{004}, there exists a bounded and one-to-one operator $V\in B(H)$ such that $\Lambda_\omega=\Theta_\omega V$, $a.e.~ [\mu]$.
By above note, $V$ can be written as $V =\frac{1}{2}(UW+UW^*)$, where $U$ is an isometry and $W$ is unitary. So $UW$ and $UW^*$ are isometries. Proposition \ref{004.4} implies that $\{\Theta_\omega UW\}_{\omega \in \Omega}$ and
 $\{\Theta_\omega UW^*\}_{\omega \in \Omega}$ are Parseval  $cg$-frames for $H$.
\end{proof}

Now, we can show each $cg$-frame as a combination of a $cg$-orthonormal basis and a $cg$-Riesz basis of $H$.

\begin{theorem}
Assuming $\{\Theta_\omega \}_{\omega \in \Omega}$ as a $cg$-orthonormal basis for $H$, Every $cg$-frame for $H$ is sum  of  a
$cg$-orthonormal basis for $H$ and a $cg$-Riesz basis for $H$.
\end{theorem}

\begin{proof}
The proof is similar to the proof of Theorem 4.2 in \cite{Ra}.

\end{proof}
%%%%%%%%%%%%%%%%%%%%%%%%%%%%%%%%%%%%%%%%%%%%%%%%%%%%%%%%%%%%%%%%%%%%%%%%%%%%%%%%%%%%%%%%%%%%%%%%%%%%%%%%%%%%%%%%%%%%%%%%%%%%%%%%%%%%%%%%%%%%%%%%%%%%%%%%%%%%%%%%%%%%%%

\bibliographystyle{amsplain}

\vspace{7mm}
\end{document}